\documentclass[11pt]{amsart}

\usepackage{amssymb,caption,enumitem,mathtools,tikz}
\usepackage[colorlinks=true,allcolors=blue]{hyperref}
\usepackage[margin=1.5in]{geometry}

\newtheorem{theorem}{Theorem}
\newtheorem{conjecture}[theorem]{Conjecture}
\newtheorem{corollary}[theorem]{Corollary}
\newtheorem{lemma}[theorem]{Lemma}
\newtheorem{proposition}[theorem]{Proposition}

\newcommand{\thistheoremname}{}
\newtheorem*{genericthm*}{\thistheoremname}
\newenvironment{namedthm*}[1]
  {\renewcommand{\thistheoremname}{#1}%
   \begin{genericthm*}}
  {\end{genericthm*}}

\theoremstyle{remark}
\newtheorem{example}[theorem]{Example}
\newtheorem*{remark}{Remark}

\DeclareMathOperator{\diag}{diag}
\DeclareMathOperator{\im}{im}
\DeclareMathOperator{\rank}{rank}

\DeclarePairedDelimiter{\abs}{\lvert}{\rvert}

\newcommand{\vd}{\mathbf{d}}
\newcommand{\vr}{\mathbf{r}}

\title[Critical groups under a generalized star-clique operation]{Critical groups of arithmetical structures\\ under a generalized star-clique operation}

\author{Alexander Diaz-Lopez}
\address[A.~Diaz-Lopez]{Department of Mathematics and Statistics, Villanova University, 800 Lancaster Ave (SAC 305), Villanova, PA 19085, USA}
\email{alexander.diaz-lopez@villanova.edu}

\author{Joel Louwsma}
\address[J.~Louwsma]{Department of Mathematics, Niagara University, Niagara University, NY 14109, USA}
\email{jlouwsma@niagara.edu}

\begin{document}

\begin{abstract}
An arithmetical structure on a finite, connected graph without loops is given by an assignment of positive integers to the vertices such that, at each vertex, the integer there is a divisor of the sum of the integers at adjacent vertices, counted with multiplicity if the graph is not simple. Associated to each arithmetical structure is a finite abelian group known as its critical group. Keyes and Reiter gave an operation that takes in an arithmetical structure on a finite, connected graph without loops and produces an arithmetical structure on a graph with one fewer vertex. We study how this operation transforms critical groups. We bound the order and the invariant factors of the resulting critical group in terms of the original arithmetical structure and critical group. When the original graph is simple, we determine the resulting critical group exactly.
\end{abstract}

\maketitle

\section{Introduction}
This paper describes the impact of a generalized star-clique operation on the critical group of an arithmetical structure. Let $G$ be a connected graph without loops with vertex set $\{v_1,v_2,\dotsc,v_n\}$, and let $A$ be the adjacency matrix of~$G$. An \emph{arithmetical structure} on~$G$ is a pair $(\vd,\vr)\in(\mathbb{Z}_{\geq0})^n\times (\mathbb{Z}_{>0})^n$ for which $(\diag(\vd)-A)\vr=\mathbf{0}$, where $\diag(\vd)$ is the diagonal matrix with the entries of~$\vd$ on the diagonal.\footnote{The effect of requiring the entries of~$\vd$ to be nonnegative rather than positive is only to allow the arithmetical structure $(\vd,\vr)=(0,1)$ on the graph with one vertex and no edges.} We consider arithmetical structures up to scaling by requiring that the greatest common divisor of the entries of~$\vr$ is~$1$. Letting $\delta_{i,j}$ denote the number of edges between vertices $v_i$ and~$v_j$, the condition $(\diag(\vd)-A)\vr=\mathbf{0}$ is equivalent to 
\[
d_ir_i=\sum_{j\in([n]\setminus\{i\})}\delta_{i,j}r_j \text{ for all }i\in[n],
\]
where $[n]$ denotes the set $\{1,2,\dotsc,n\}$. Thinking of an arithmetical structure as a labeling of the vertices of~$G$ by the entries of~$\vr$, this means that, at each vertex, the integer there is a divisor of the sum of the integers at adjacent vertices, counted with multiplicity if the graph is not simple. 

Arithmetical structures were first introduced by Lorenzini~\cite{L89}; matrices of the form $(\diag(\vd)-A)$ arise in algebraic geometry as intersection matrices of degenerating curves. The simplest example of an arithmetical structure is given by letting $\vd$ record the degrees of the vertices of~$G$ and letting all entries of~$\vr$ be~$1$. In this case the matrix $(\diag(\vd)-A)$ is the \emph{Lapacian matrix} of~$G$, and the corresponding arithmetical structure is known as the \emph{Laplacian arithmetical structure}. Thinking of arithmetical structures as generalizing this example, we use $L$ as a shorthand for the matrix $(\diag(\vd)-A)$ throughout this paper and denote its entries by $\ell_{i,j}$.

For any arithmetical structure $(\vd,\vr)$, Lorenzini \cite[Proposition~1.1]{L89} showed that the matrix~$L$ has rank $n-1$. Regarding $L$ as a $\mathbb{Z}$-linear transformation from $\mathbb{Z}^n$ to $\mathbb{Z}^n$, the cokernel $\mathbb{Z}^n/\im(L)$ is therefore of the form $\mathbb{Z}\oplus \mathcal{K}(G;\vd,\vr)$ for some finite abelian group $\mathcal{K}(G;\vd,\vr)$ known as the \emph{critical group} of $(G,\vd,\vr)$. In algebraic geometry, the critical group appears as the group of components of the N\'{e}ron model of the Jacobian of the generic curve. In the case of the Laplacian arithmetical structure, the critical group is also known as the \emph{sandpile group}, which has been of much interest in recent years; see, for example, \cite{CP,K19}. Critical groups of general arithmetical structures are understood in relatively few cases. Braun et al.~\cite{B18} showed that the critical group of any arithmetical structure on a path graph is the trivial group and that the critical group of any arithmetical structure on a cycle graph is cyclic. Archer et al.~\cite{A20} showed that the critical group of any arithmetical structure on a bident is cyclic. Lorenzini~\cite{L24} showed that every finite, connected graph admits an arithmetical structure with trivial critical group.

In \cite[Section~1.8]{L89}, Lorenzini defined a \emph{blowup} operation on graphs that in some cases would convert an arithmetical structure on a graph~$G$ into an arithmetical structure on a larger graph while preserving the critical group. Corrales and Valencia~\cite{C18} explicitly wrote down an instance of the blowup operation that converts a clique subgraph to a star subgraph. In the case of an arithmetical structure $(\vd,\vr)$ on~$G$, the inverse of the blowup operation is called the \emph{star-clique} operation and replaces a star subgraph of~$G$ with entry of~$\vd$ equal to~$1$ at the center by a clique subgraph to obtain an arithmetical structure on a graph~$G'$ with one fewer vertex. In \cite{A20}, \cite{B18}, and~\cite{GW}, a special case of the star-clique operation, known as \emph{smoothing}, was essential in enumerating arithmetical structures on path graphs, cycle graphs, and certain paths with a doubled edge, bounding the number of arithmetical structures on bidents, and classifying critical groups of arithmetical structures on cycle graphs and bidents. Keyes and Reiter~\cite{KR} defined an operation generalizing the star-clique operation that can be performed on any graph (not necessarily simple) and at any vertex (without requiring that the corresponding entry of~$\vd$ is~$1$). They used it to obtain an upper bound on the number of arithmetical structures on a graph that depends only on the number of vertices and number of edges of the graph.

Given an arithmetical structure $(\vd,\vr)$ on~$G$, let $(G',\vd',\vr')$ denote the result of applying the operation defined by Keyes and Reiter to $(G,\vd,\vr)$ at the vertex~$v_n$; by permuting the labels of the vertices, this operation can be performed at any vertex. We let $A'$ be the adjacency matrix of~$G'$ and use $L'$ as a shorthand for the matrix $(\diag(\vd')-A')$, denoting its entries by $\ell'_{i,j}$. In \cite[Remark~2.5]{KR}, Keyes and Reiter suggest studying the relationship between the critical groups $\mathcal{K}(G;\vd,\vr)$ and $\mathcal{K}(G';\vd',\vr')$, which we do by relating their invariant factors. Recall that the invariant factor decomposition of finite abelian groups allows one to write
\[
\mathcal{K}(G;\vd,\vr)\cong \bigoplus_{k=1}^{n-1} \mathbb{Z}/\alpha_{k}\mathbb{Z}\qquad \text { and } \qquad \mathcal{K}(G';\vd',\vr')\cong\bigoplus_{k=1}^{n-2}\mathbb{Z}/\alpha'_k\mathbb{Z},
\]
where $\alpha_k \mid \alpha_{k+1}$ for all $k \in [n-2]$ and $\alpha'_k\mid \alpha'_{k+1}$ for all $k\in[n-3]$. The \emph{invariant factors}~$\alpha_k$ (respectively, $\alpha'_k$) are uniquely determined up to sign, can be read off explicitly from the Smith normal form of~$L$ (resp.~$L'$), and are determined by greatest common divisors of appropriately sized minors of~$L$ (resp.~$L'$). In this paper, we choose the invariant factors to be positive. 

Our approach stems from Lemma~\ref{lem:Chio}, which relates minors of~$L'$ to minors of~$L$ that include the last row and the last column. Motivated by this connection, in Section~\ref{sec:minors} we prove several divisibility properties of minors of matrices with integer entries. The primary results of this section (Proposition~\ref{prop:Dk*dividesDk+1*}, Lemma~\ref{lem:Dk*upperbound}, and Proposition~\ref{prop:D1D2*dividesD1*D2}) do not hold at the level of determinantal ideals but rely on factorization properties of the integers. These results may be of independent interest to some readers.

Section~\ref{sec:mainresults} applies the results of Section~\ref{sec:minors} to the study of critical groups of arithmetical structures. Our first main result is Theorem~\ref{thm:DkL'}, which constrains the greatest common divisor of all $k\times k$ minors of~$L'$. As a corollary, we obtain the following restrictions on the order of $\mathcal{K}(G';\vd',\vr')$ in terms of $(G,\vd,\vr)$ and its critical group. Here, $g_n(L)$ denotes the greatest common divisor of the entries of row~$n$ of~$L$.

\begin{namedthm*}{Corollary~\ref{cor:orderbound}}
The following statements hold:
\begin{enumerate}[label=\textup{(\alph*)},ref=\textup{\alph*}]
\item $d_n^{n-3}\abs{\mathcal{K}(G;\vd,\vr)}$ divides $\abs{\mathcal{K}(G';\vd',\vr')}$,
\item $\abs{\mathcal{K}(G';\vd',\vr')}$ divides $(g_n(L))^2d_n^{n-3}\abs{\mathcal{K}(G;\vd,\vr)}$, and
\item $d_n^{n-3}\abs{\mathcal{K}(G;\vd,\vr)}\leq \abs{\mathcal{K}(G';\vd',\vr')}\leq (g_n(L))^2d_n^{n-3}\abs{\mathcal{K}(G;\vd,\vr)}$.
\end{enumerate}
\end{namedthm*}

Example~\ref{ex:nonsimple} shows that, when $g_n(L)\neq1$, the lower and upper bounds of Corollary~\ref{cor:orderbound}(\ref{part:orderbound}) are each sometimes achieved, so in this sense Corollary~\ref{cor:orderbound} is optimal.

Our next main results, Proposition~\ref{prop:alpha'1} and Theorem~\ref{thm:alpha'k}, constrain $\alpha'_k$ for all $k \in [n-2]$. In the special case where $g_n(L)=1$, and in particular when $G$ is a simple graph, these results yield the following corollary that shows $\mathcal{K}(G';\vd',\vr')$ is completely determined by $\mathcal{K}(G,\vd,\vr)$ and~$d_n$.

\begin{namedthm*}{Corollary~\ref{cor:gcd=1}}
If $g_n(L)=1$ \textup{(}for example, if $G$ is a simple graph\textup{)}, then $\alpha'_1=\alpha_2$ and $\alpha'_k=d_n\alpha_{k+1}$ for all $k\in \{2,3,\dotsc, n-2\}$. In particular,
\[
\mathcal{K}(G';\vd',\vr')\cong\mathbb{Z}/\alpha_{2}\mathbb{Z}\oplus \Biggl(\bigoplus_{k=2}^{n-2}\mathbb{Z}/d_n\alpha_{k+1}\mathbb{Z}\Biggr).
\]
\end{namedthm*}

When $d_n=1$, Corollary~\ref{cor:gcd=1} implies that the star-clique operation preserves the critical group, as was observed by Lorenzini \cite[Section~1.8]{L89}. When $g_n(L)\neq1$, the critical group $\mathcal{K}(G';\vd',\vr')$ is not completely determined by $\mathcal{K}(G;\vd,\vr)$ and~$d_n$, as shown by Example~\ref{ex:nonsimple}.

We conclude with Conjecture~\ref{conj:minors}, which is stated only in terms of minors of matrices. Showing this conjecture would allow improvement of the upper bound in Corollary~\ref{cor:orderbound} in the case where $(g_n(L))^2\alpha_1$ does not divide~$d_n$.

\section{Preliminaries}\label{sec:preliminaries}

In this section, we describe the operation defined by Keyes and Reiter in~\cite{KR}, recall some facts about Smith normal forms of matrices that will be useful in finding invariant factor decompositions of critical groups, and give examples of the operation and the computation of critical groups.

\subsection{The Keyes--Reiter generalized star-clique operation}\label{subsec:operation}
Given an arithmetical structure $(\vd,\vr)$ on a graph~$G$ with vertex set $V=\{v_1,v_2,\dotsc, v_n\}$, the operation defined by Keyes and Reiter~\cite{KR} can be applied to any vertex $v_i \in V$ to produce an arithmetical structure $(\vd',\vr')$ on a graph~$G'$ with vertex set $V\setminus\{v_i\}$. For simplicity, in this paper we relabel the vertices so that the operation is always performed at vertex~$v_n$. To describe the operation, we first construct the resulting graph~$G'$ and then construct the arithmetical structure $(\vd',\vr')$ on~$G'$.

Let $A$ be the adjacency matrix of~$G$, i.e.\ the matrix whose entries $\delta_{i,j}$ record the number of edges between $v_i$ and~$v_j$ in~$G$. Construction~2.1 of~\cite{KR} gives that the number of edges between $v_i$ and~$v_j$ in~$G'$ is $\delta_{i,j}d_n+\delta_{i,n}\delta_{n,j}$ for all $i,j\in[n-1]$ with $i\neq j$ and that $G'$ has no loops, so we have
\[
\delta'_{i,j}=\begin{cases}
\delta_{i,j}d_n+\delta_{i,n}\delta_{n,j}&\text{if }i\neq j\\
0&\text{if }i=j.
\end{cases}
\]
As in \cite[Remark~2.2]{KR}, the edges of~$G'$ may be thought of as arising through a two-step process. First, each edge of~$G$ not incident to~$v_n$ becomes $d_n$~edges with the same endpoints in~$G'$. Second, for any pair of vertices $v_i$ and~$v_j$ with $i,j\in[n-1]$ and $i\neq j$, we add $\delta_{i,n}\delta_{n,j}$ edges between $v_i$ and~$v_j$ in~$G'$. Figure~\ref{fig:example1} provides an example of the operation on a simple graph at the vertex~$v_7$ with $d_7=3$; the edges created in the first step are shown with solid lines and those created in the second step are shown with dashed lines. 

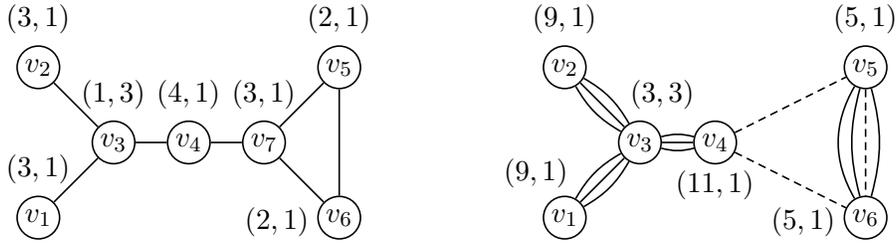
\begin{figure}
    \centering
    \begin{tikzpicture}
    \tikzstyle{every node}=[draw,circle,minimum size=4pt,inner sep=1.5pt]
    \draw [semithick] (1,1) node (v3)[label={[yshift=-0.2cm] $(1,3)$}] {$v_{3}$};
    \draw [semithick] (0,0) node (v1) [label={[yshift=-0.2cm] $(3,1)$}] {$v_{1}$} edge  (v3);
    \draw [semithick] (0,2) node (v2)[label={[yshift=-0.2cm] $(3,1)$}] {$v_{2}$} edge (v3);
    \draw [semithick] (2,1) node (v4)[label={[yshift=-0.2cm] $(4,1)$}] {$v_{4}$} edge (v3);
    \draw [semithick] (3,1) node (v7)[label={[yshift=-0.2cm] $(3,1)$}] {$v_{7}$} edge (v4);
    \draw [semithick] (4,2) node (v5)[label={[yshift=-0.2cm] $(2,1)$}] {$v_{5}$} edge (v7);
    \draw [semithick] (4,0) node (v6)[label={[yshift=-0.9cm,xshift=-.825cm] $(2,1)$}] {$v_{6}$} edge (v7);
    \draw [semithick] (v5) edge (v6);
    \end{tikzpicture}\qquad\qquad
    \begin{tikzpicture}
    \tikzstyle{every node}=[draw,circle,minimum size=4pt,inner sep=1.5pt]
    \draw [semithick] (0,0) node (v1) [label={[yshift=-0.25cm,xshift=-.38cm] $(9,1)$}] {$v_{1}$};
    \draw [semithick] (0,2) node (v2)[label={[yshift=-0.2cm] $(9,1)$}]{$v_{2}$} ;
    \draw [semithick] (2,1) node (v4)[label={[yshift=-1.5cm] $(11,1)$}] {$v_{4}$};
    \draw [semithick] (4,2) node (v5)[label={[yshift=-0.2cm] $(5,1)$}]{$v_{5}$};
    \draw [semithick] (4,0) node (v6)[label={[yshift=-0.9cm,xshift=-.825cm] $(5,1)$}]  {$v_{6}$};
    \draw [semithick] (1,1) node (v3)[label={[yshift=-0.2cm,xshift=.3cm] $(3,3)$}]{$v_{3}$};
    \draw [semithick] (v3) edge [bend left=15] (v1);
    \draw [semithick] (v3) edge[bend right=15] (v1);
    \draw [semithick] (v3) edge (v1);
    \draw [semithick] (v3) edge (v2);
    \draw [semithick] (v3) edge[bend left=15] (v2);
    \draw [semithick] (v3) edge[bend right=15] (v2);
    \draw [semithick] (v3) edge (v4);
    \draw [semithick] (v3) edge[bend left=15] (v4);
    \draw [semithick] (v3) edge[bend right=15] (v4);
    \draw [semithick] (v6) edge[bend left] (v5);
    \draw [semithick] (v6) edge[bend right=15] (v5);
    \draw [semithick] (v6) edge[bend left=15] (v5);
    \draw [densely dashed,semithick] (v5) edge (v6);
    \draw [densely dashed,semithick] (v4) edge (v5);
    \draw [densely dashed,semithick] (v4) edge (v6);
    \end{tikzpicture}
    \caption{On the left is a graph~$G$ with vertices labeled by $(d_i,r_i)$, the entries of an arithmetical structure $(\vd,\vr)$. Performing the operation at~$v_7$ produces the graph~$G'$ on the right with vertices labeled by entries of the arithmetical structure $(\vd',\vr')$.\label{fig:example1}}
\end{figure}

As in \cite[Lemma~2.3]{KR}, setting $d'_i=d_id_n-\delta_{i,n}\delta_{n,i}$ and $r'_i=r_i/\gcd_{j\in[n-1]}(r_j)$ for all $i\in[n-1]$ gives an arithmetical structure $(\vd',\vr')$ on~$G'$. Let $A'$ be the matrix with entries $\delta'_{i,j}$, that is, the adjacency matrix of~$G'$. We have that
\[
L=\diag(\vd)-A=\scriptsize{\begin{bmatrix}
d_1 & -\delta_{1,2} & \cdots & -\delta_{1,n}\\\
-\delta_{2,1} & d_2 & \ddots & \vdots\\
\vdots & \ddots & \ddots & -\delta_{n-1,n}\\
-\delta_{n,1} & \cdots & -\delta_{n,n-1} & d_n
\end{bmatrix}}
\] 
and
\[
L'=\diag(\vd')-A'=\scriptsize{\begin{bmatrix}
d'_1 & -\delta'_{1,2} & \cdots & -\delta'_{1,n-1}\\\
-\delta'_{2,1} & d'_2 & \ddots & \vdots\\
\vdots & \ddots & \ddots & -\delta'_{n-2,n-1}\\
-\delta'_{n-1,1} & \cdots & -\delta'_{n-1,n-2} & d'_{n-1}
\end{bmatrix}}.
\]
From the above discussion, $L'$ is equal to 
\[
\resizebox{\textwidth}{!}{
$\begin{bsmallmatrix}
d_{1}d_{n}-\delta _{1,n}\delta _{n,1}&-\delta _{1,2}d_{n}-\delta _{1,n}\delta _{n,2}&\text{\dots} &-\delta _{1,n-1}d_{n}-\delta _{1,n}\delta _{n,n-1}\\
-\delta _{2,1}d_{n}-\delta _{2,n}\delta _{n,1}&d_{2}d_{n}-\delta _{2,n}\delta _{n,2}&\scriptsize{\ddots} &\scriptsize{\vdots} \\
\scriptsize{\vdots} &\scriptsize{\ddots} &\scriptsize{\ddots} &-\delta _{n-2,n-1}d_{n}-\delta _{n-2,n}\delta _{n,n-1}\\
-\delta _{n-1,1}d_{n}-\delta _{n-1,n}\delta _{n,1}&\text{\dots} &-\delta _{n-1,n-2}d_{n}-\delta _{n-1,n}\delta _{n,n-2}&d_{n-1}d_{n}-\delta _{n-1,n}\delta _{n,n-1}
\end{bsmallmatrix}$.}
\]
Observe that the entries of the matrix~$L'$ are $2\times 2$ minors of~$L$ that include the last row and the last column. More specifically, if $(\ell_{i,j})_{i,j \in [n]}$ denote the entries of~$L$ and $(\ell'_{i,j})_{i,j \in[n-1]}$ denote the entries of~$L'$, then $\ell'_{i,j}=\ell_{i,j}\ell_{n,n}-\ell_{i,n}\ell_{n,j}$ for all $i,j \in [n-1]$.

Subsection \ref{subsec:examples} gives examples of the operation.

\subsection{Computing critical groups}\label{subsec:criticalgroups}
We next recall some facts about Smith normal forms of matrices and their connection to critical groups; see \cite{GK,L08,SNF} for more details.

Given an $m\times n$ matrix~$B$ with integer entries, there exists an $m\times m$ matrix~$S$ and an $n\times n$ matrix~$T$, both invertible and with integer entries, such that $SBT$ is of the form
\[
\scriptsize{\begin{bmatrix}
\alpha_1 & 0 & \cdots & \cdots & \cdots & \cdots & 0\\
0 & \alpha_2 & \ddots & & & & \vdots\\
\vdots & \ddots & \ddots & \ddots & & & \vdots\\
\vdots & & \ddots & \alpha_{t} & \ddots & & \vdots\\
\vdots & & & \ddots & 0 & \ddots & \vdots\\
\vdots & & & & \ddots & \ddots & 0\\
0 & \cdots & \cdots & \cdots & \cdots & 0 & 0
\end{bmatrix},}
\]
where $t=\rank(B)$, the $\alpha_k$ are positive integers, and $\alpha_k$ divides $\alpha_{k+1}$ for all $k\in[t-1]$. This diagonal matrix is the \emph{Smith normal form} of~$B$, and the entries~$\alpha_k$ are the \emph{invariant factors} of~$B$. Thinking of~$B$ as a linear map $\mathbb{Z}^n \to \mathbb{Z}^m$, the cokernel $\mathbb{Z}^m/\im(B)$ is isomorphic to 
\[
\mathbb{Z}^{m-t}\oplus\Biggl(\bigoplus_{k=1}^{t} \mathbb{Z}/\alpha_{k}\mathbb{Z}\Biggr).
\]
Letting $D_{k}(B)$ be the greatest common divisor of all $k\times k$ minors of~$B$ and defining $D_0(B)=1$, one has that $\alpha_{k}=D_{k}(B)/D_{k-1}(B)$ for all $k\in[t]$. The order of the torsion part of the cokernel $\mathbb{Z}^m/\im(B)$ is then
\[
\prod_{k=1}^{t} \alpha_{k}=\prod_{k=1}^{t}\frac{D_{k}(B)}{D_{k-1}(B)}=D_{t}(B).
\]

In the case of~$L$, the square matrix $(\diag(\vd)-A)$ associated to an arithmetical structure on a graph with $n$~vertices, Lorenzini \cite[Proposition~1.1]{L89} showed that $\rank(L)=n-1$. Hence, the critical group of $(G,\vd,\vr)$ is
\[
\mathcal{K}(G;\vd,\vr)\cong \bigoplus_{k=1}^{n-1} \mathbb{Z}/\alpha_{k}\mathbb{Z}
\]
and $\abs{\mathcal{K}(G;\vd,\vr)}=D_{n-1}(L)$. Performing the operation on $(G,\vd,\vr)$ at the vertex~$v_n$, one gets the arithmetical structure $(\vd',\vr')$ on~$G'$ with critical group
\[
\mathcal{K}(G';\vd',\vr')\cong \bigoplus_{k=1}^{n-2} \mathbb{Z}/\alpha'_{k}\mathbb{Z},
\]
where the~$\alpha'_k$ are the invariant factors of $L'=\diag{(\vd}')-A'$ and $A'$ is the adjacency matrix of~$G'$. The primary purpose of this paper is to study the relationship between the invariant factors $\{\alpha_k\}_{k=1}^{n-1}$ of~$L$ and the invariant factors $\{\alpha'_k\}_{k=1}^{n-2}$ of~$L'$. We obtain results relating these factors in Proposition~\ref{prop:alpha'1}, Theorem~\ref{thm:alpha'k}, and Corollary~\ref{cor:gcd=1}.

\subsection{Examples}\label{subsec:examples}
We end this section with examples of the Keyes--Reiter generalized star-clique operation and the computation of critical groups. The first example involves an arithmetical structure on a simple graph and illustrates the results of Corollary~\ref{cor:gcd=1}. 

\begin{example}\label{ex:examplesimple}
Consider the graph~$G$ on the left of Figure~\ref{fig:example1} with the arithmetical structure defined by $\vd=(3,3,1,4,2,2,3)$ and $\vr=(1,1,3,1,1,1,1)$. The Smith normal form of~$L$ is
\[
\resizebox{\textwidth}{!}{
$\begin{bsmallmatrix}
1 & 0 & 0 & 0 & 0 & 0 & 0\\
0 & 1 & 0 & 0 & 0 & 0 & 0\\
0 & 0 & 1 & 0 & 0 & 0 & 0\\
0 & 0 & 0 & 1 & 0 & 0 & 0\\
0 & 0 & 0 & 0 & 3 & 0 & 0\\
0 & 0 & 0 & 0 & 0 & 3 & 0\\
0 & 0 & 0 & 0 & 0 & 0 & 0
\end{bsmallmatrix}
=\begin{bsmallmatrix}
0 & 0 & 1 & 0 & 0 & 0 & 0\\
0 & 0 & 0 & 1 & 0 & 0 & 0\\ 
0 & 0 & 0 & 0 & 0 & 0 & 1\\ 
0 & 1 & 0 & 0 & 0 & 0 & 0\\ 
0 & -2 & 0 & -1 & 0 & 1 & -1\\ 
0 & -2 & 0 & -1 & 1 & 0 & -1\\ 
-1 & -1 & -3 & -1 & -1 & -1 & -1
\end{bsmallmatrix}
\begin{bsmallmatrix}
3 & 0 & -1 & 0 & 0 & 0 & 0\\
0 & 3 & -1 & 0 & 0 & 0 & 0\\ 
-1 & -1 & 1 & -1 & 0 & 0 & 0\\ 
0 & 0 & -1 & 4 & 0 & 0 & -1\\ 
0 & 0 & 0 & 0 & 2 & -1 & -1\\ 
0 & 0 & 0 & 0 & -1 & 2 & -1\\ 
0 & 0 & 0 & -1 & -1 & -1 & 3
\end{bsmallmatrix}
\begin{bsmallmatrix}
-1 & -4 & -5 &-7 & -5 & -5 & 1\\
 0 & -3 &-4 & -5 & -4 & -4 &1\\
 0 & -9 & -12 &-16 &-12 &-12 & 3\\
 0 & -2 &-3 & -4 & -3 & -3 & 1\\
 0 & 1 & 1 &2 & 1 & 2 &1\\
 0 & 1 & 1 &2 & 2 & 1 &1\\
 0 & 0 & 0 &0 & 0 & 0 & 1
\end{bsmallmatrix}$.\label{eq:before}}
\]
Thus $\mathcal{K}(G;\vd,\vr)\cong\mathbb{Z}/3\mathbb{Z}\oplus\mathbb{Z}/3\mathbb{Z}$. Performing the operation at the vertex~$v_7$, we obtain the graph~$G'$ on the right of Figure~\ref{fig:example1} with the arithmetical structure defined by $\vd'=(9,9,3,11,5,5)$ and $\vr'=(1,1,3,1,1,1)$. The matrix~$L'$ corresponding to $(G',\vd',\vr')$ has Smith normal form
\[
\resizebox{\textwidth}{!}{
$\begin{bsmallmatrix}
1 & 0 & 0 & 0 & 0 & 0\\
0 & 3 & 0 & 0 & 0 & 0\\
0 & 0 & 3 & 0 & 0 & 0\\
0 & 0 & 0 & 9 & 0 & 0\\
0 & 0 & 0 & 0 & 9 & 0\\
0 & 0 & 0 & 0 & 0 & 0 
\end{bsmallmatrix}
=\begin{bsmallmatrix}
0 & 0 & 0 & 1 & 0 & 0\\
0 & 0 & 1 & 0 & 0 & 0\\
0 & 1 & 0 & 0 & 0 & 0\\
0 & -2 & 0 & -4 & 0 & 1\\
0 & -2 & 0 & -4 & 1 & 0\\
-1 & -1 & -3 & -1 & -1 & -1
\end{bsmallmatrix} 
\begin{bsmallmatrix}
9 & 0 & -3 & 0 & 0 & 0\\ 
0 & 9 & -3 & 0 & 0 & 0\\ 
-3 & -3 & 3 & -3 & 0 & 0\\ 
0 & 0 & -3 & 11 & -1 & -1\\ 
0 & 0 & 0 & -1 & 5 & -4\\
0 & 0 & 0 & -1 & -4 & 5
\end{bsmallmatrix}
\begin{bsmallmatrix}
-6 & -1 & -9 & -7 & -6 & 1\\
-5 & 0 & -7 & -6 & -5 & 1\\
-15 & 0 & -22 & -18 & -15 & 3\\
-4 & 0 & -6 & -5 & -4 & 1\\
0 & 0 & 0 & -1 & 1 & 1\\
0 & 0 & 0 & 0 & 0 & 1
\end{bsmallmatrix}$.\label{eq:after}}
\]
Thus $\mathcal{K}(G';\vd',\vr')\cong\mathbb{Z}/3\mathbb{Z}\oplus \mathbb{Z}/3\mathbb{Z}\oplus\mathbb{Z}/9\mathbb{Z}\oplus\mathbb{Z}/9\mathbb{Z}$. This is consistent with Corollary~\ref{cor:gcd=1}, which guarantees that, for simple graphs, if $(\alpha_1,\alpha_2,\dotsc,\alpha_{n-1},0)$ are the diagonal entries of the Smith normal form of~$L$ then $(\alpha_2,d_n\alpha_3,\dotsc,d_n\alpha_{n-1},0)$ are the diagonal entries in the Smith normal form of~$L'$; this is the case here with $d_n=d_7=3$. 
\end{example}

The case of non-simple graphs is more intriguing. The next example gives two distinct arithmetical structures on the same non-simple graph with the same critical group and the same value of~$d_n$ but with different critical groups after performing the operation, thus showing that $\mathcal{K}(G';\vd',\vr')$ is not completely determined by $\mathcal{K}(G;\vd,\vr)$ and~$d_n$. It also shows that each of the bounds of Corollary~\ref{cor:orderbound}(\ref{part:orderbound}) can be achieved when $g_n(L)\neq1$.

\begin{example}\label{ex:nonsimple}
Consider the graph~$G$ on the left of Figure~\ref{fig:example2a} with the arithmetical structure defined by $\vd=(8,10,4,8)$ and $\vr=(1,3,5,2)$. The matrix~$L$ has Smith normal form
\[
\begin{bsmallmatrix}
1 & 0 & 0 & 0\\
0 & 1 & 0 & 0\\
0 & 0 & 24 & 0\\
0 & 0 & 0 & 0
\end{bsmallmatrix}
=\begin{bsmallmatrix}
1 & 3 & 6 & 2\\
1 & 4 & 4 & -2\\
0 & 2 & -2 & -9\\
1 & 3 & 5 & 2
\end{bsmallmatrix}
\begin{bsmallmatrix}
8 & -1 & -1 & 0\\
-1 & 10 & -5 & -2\\ 
-1 & -5 & 4 & -2\\ 
0 & -2 & -2 & 8 
\end{bsmallmatrix}
\begin{bsmallmatrix}
-1 & -4 & 43 & 1\\
0 & 0 & -1 & 3\\
0 & -1 & 9 & 5\\
0 & 0 & -1 & 2
\end{bsmallmatrix},
\] 
so $\mathcal{K}(G;\vd,\vr)\cong\mathbb{Z}/24\mathbb{Z}$. Performing the operation at the vertex~$v_4$, we get an arithmetical structure defined by $\vd'=(64,76,28)$ and $\vr'=(1,3,5)$ on the graph on the right of Figure~\ref{fig:example2a}, where the label on an edge indicates the multiplicity of edges between the corresponding vertices. The matrix~$L'$ has Smith normal form
\[
\begin{bsmallmatrix}
4 & 0 & 0\\
0 & 48 & 0\\
0 & 0 & 0
\end{bsmallmatrix}
=\begin{bsmallmatrix}
1 & 3 & 6\\
1 & 4 & -2\\
1 & 3 & 5
\end{bsmallmatrix}
\begin{bsmallmatrix}
64 & -8 & -8\\
-8 & 76 & -44\\ 
-8 & -44 & 28
\end{bsmallmatrix}
\begin{bsmallmatrix}
-2 & -2 & 1\\
-1 & 1 & 3\\
-2 & 1 & 5
\end{bsmallmatrix},
\] 
so $\mathcal{K}(G';\vd',\vr')\cong\mathbb{Z}/4\mathbb{Z}\oplus\mathbb{Z}/48\mathbb{Z}$. This achieves the lower bound of Corollary~\ref{cor:orderbound}(\ref{part:orderbound}), as 
\[
d_4\cdot \abs{\mathcal{K}(G;\vd,\vr)}=8\cdot 24=192=\abs{\mathcal{K}(G';\vd',\vr')}.
\]

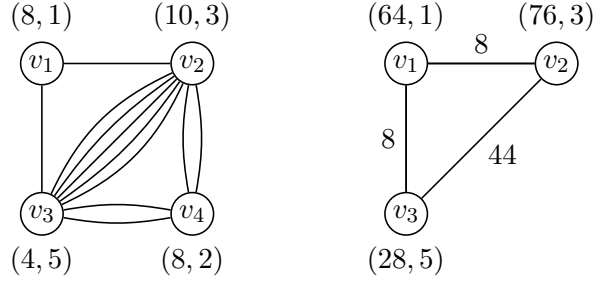
\begin{figure}
    \centering
    \begin{tikzpicture}
    \tikzstyle{style1}=[draw,circle,minimum size=4pt,inner sep=1.5pt]
    \draw [semithick] (2,0) node[style1] (v4)[label={below: $(8,2)$}] {$v_{4}$};
    \draw [semithick] (2,2) node[style1] (v2) [label={above: $(10,3)$}] {$v_{2}$};
    \draw [semithick] (0,2) node[style1] (v1) [label={above: $(8,1)$}] {$v_{1}$} edge (v2);
    \draw [semithick] (0,0) node[style1] (v3) [label={below: $(4,5)$}] {$v_{3}$} edge (v1);
    \draw [semithick] (v2) edge[bend right=10] (v3);
    \draw [semithick] (v2) edge[bend left=20] (v3);
    \draw [semithick] (v2) edge[bend right=20] (v3);
    \draw [semithick] (v2) edge[bend left=10] (v3);
    \draw [semithick] (v2) edge (v3);
    \draw [semithick] (v3) edge[bend right=10] (v4);
    \draw [semithick] (v3) edge[bend left=10] (v4);
    \draw [semithick] (v2) edge[bend right=10] (v4);
    \draw [semithick] (v2) edge[bend left=10] (v4);
    \end{tikzpicture}\qquad \qquad
    \begin{tikzpicture}
    \tikzstyle{style1}=[draw,circle,minimum size=4pt,inner sep=1.5pt]
    \draw [semithick] (2,2) node[style1] (v2) [label={above: $(76,3)$}] {$v_{2}$};
    \draw [semithick] (0,2) node[style1] (v1) [label={above: $(64,1)$}] {$v_{1}$} edge (v2);
    \draw [semithick] (0,0) node[style1] (v3) [label={below: $(28,5)$}] {$v_{3}$} edge (v1);
    \draw [semithick] (v2) edge node[below right=-0.05cm] {44} (v3);
    \draw [semithick] (v1) edge node[above] {8} (v2);
     \draw [semithick] (v1) edge node[left] {8} (v3);
    \end{tikzpicture}
    \caption{An arithmetical structure on a non-simple graph labeled by the entries of $(\vd,\vr)$ and the result of performing the operation at the vertex~$v_4$.}\label{fig:example2a}
\end{figure}
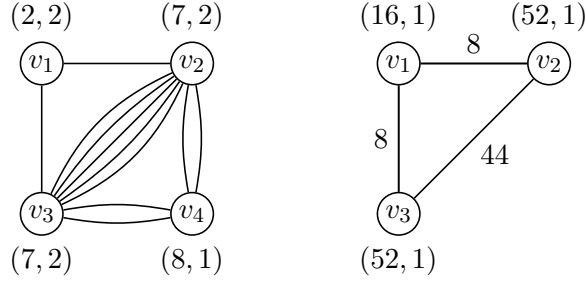
\begin{figure}
    \centering
    \begin{tikzpicture}
    \tikzstyle{style1}=[draw,circle,minimum size=4pt,inner sep=1.5pt]
    \draw [semithick] (2,0) node[style1] (v4)[label={below: $(8,1)$}] {$v_{4}$};
    \draw [semithick] (2,2) node[style1] (v2) [label={above: $(7,2)$}] {$v_{2}$};
    \draw [semithick] (0,2) node[style1] (v1) [label={above: $(2,2)$}] {$v_{1}$} edge (v2);
    \draw [semithick] (0,0) node[style1] (v3) [label={below: $(7,2)$}] {$v_{3}$} edge (v1);
    \draw [semithick] (v2) edge[bend right=10] (v3);
    \draw [semithick] (v2) edge[bend left=20] (v3);
    \draw [semithick] (v2) edge[bend right=20] (v3);
    \draw [semithick] (v2) edge[bend left=10] (v3);
    \draw [semithick] (v2) edge (v3);
    \draw [semithick] (v3) edge[bend right=10] (v4);
    \draw [semithick] (v3) edge[bend left=10] (v4);
    \draw [semithick] (v2) edge[bend right=10] (v4);
    \draw [semithick] (v2) edge[bend left=10] (v4);
    \end{tikzpicture}\qquad \qquad
    \begin{tikzpicture}
    \tikzstyle{style1}=[draw,circle,minimum size=4pt,inner sep=1.5pt]
    \draw [semithick] (2,2) node[style1] (v2) [label={above: $(52,1)$}] {$v_{2}$};
    \draw [semithick] (0,2) node[style1] (v1) [label={above: $(16,1)$}] {$v_{1}$} edge (v2);
    \draw [semithick] (0,0) node[style1] (v3) [label={below: $(52,1)$}] {$v_{3}$} edge (v1);
    \draw [semithick] (v2) edge node[below right=-0.05cm] {44} (v3);
    \draw [semithick] (v1) edge node[above] {8} (v2);
     \draw [semithick] (v1) edge node[left] {8} (v3);
    \end{tikzpicture}
    \caption{Another arithmetical structure on the same graph as in Figure~\ref{fig:example2a} and with the same value of~$d_4$ and the result of performing the operation at the vertex~$v_4$.}\label{fig:example2b} 
\end{figure}

For the same graph~$G$, consider the arithmetical structure given by $\vd=(2,7,7,8)$ and $\vr=(2,2,2,1)$, as shown on the left of Figure~\ref{fig:example2b}. The matrix~$L$ has Smith normal form
\[
\begin{bsmallmatrix}
1 & 0 & 0 & 0\\
0 & 1 & 0 & 0\\
0 & 0 & 24 & 0\\
0 & 0 & 0 & 0
\end{bsmallmatrix}
=\begin{bsmallmatrix}
0 & 0 & 1 & 0\\
1 & 1 & 1 & 1\\
0 & 1 & -1 & -6\\
2 & 2 & 2 & 1
\end{bsmallmatrix}
\begin{bsmallmatrix}
2 & -1 & -1 & 0\\
-1 & 7 & -5 & -2\\ 
-1 & -5 & 7 & -2\\ 
0 & -2 & -2 & 8
\end{bsmallmatrix}
\begin{bsmallmatrix}
-1 & -7 & -12 & 2\\
0 & 0 & 1 & 2\\
0 & -1 & -1 & 2\\
0 & 0 & 0 & 1
\end{bsmallmatrix},
\]
so $\mathcal{K}(G;\vd,\vr)\cong\mathbb{Z}/24\mathbb{Z}$. Performing the operation at the vertex~$v_4$, we get the arithmetical structure defined by $\vd'=(16,52,52)$ and $\vr'=(1,1,1)$ on the graph on the right of Figure~\ref{fig:example2b}. The matrix~$L'$ has Smith normal form
\[
\begin{bsmallmatrix}
4 & 0 & 0\\
0 & 192 & 0\\
0 & 0 & 0
\end{bsmallmatrix}
=\begin{bsmallmatrix}
-1 & -1 & 0\\
25 & 26 & 0\\
1 & 1 & 1
\end{bsmallmatrix}
\begin{bsmallmatrix}
16& -8& -8\\
-8& 52& -44\\ 
-8& -44& 52
\end{bsmallmatrix}
\begin{bsmallmatrix}
-6 & -11 & 1\\
1 & 2 & 1\\
0 & 0 & 1
\end{bsmallmatrix},
\]
so $\mathcal{K}(G';\vd',\vr')\cong\mathbb{Z}/4\mathbb{Z}\oplus \mathbb{Z}/192\mathbb{Z}$. This achieves the upper bound of Corollary~\ref{cor:orderbound}(\ref{part:orderbound}), as
\[
d_4\cdot (g_n(L))^2\cdot \abs{\mathcal{K}(G;\vd,\vr)}=8\cdot 2^2\cdot 24=768=\abs{\mathcal{K}(G';\vd',\vr')}.
\]
\end{example}

\section{Divisibility properties of minors of matrices}\label{sec:minors}

In this section, we prove several divisibility properties of minors of matrices with integer entries that will be applied to matrices of the form $(\diag(\vd)-A)$ in Section~\ref{sec:mainresults}. This section does not use arithmetical structures or the operation discussed in Section~\ref{sec:preliminaries}, and it may thus be read independently from the rest of the paper.

Given an $m\times n$ matrix~$B$ and subsets $I\subseteq[m]$ and $J\subseteq[n]$ with $\abs{I}=\abs{J}=k$, let $B_{I,J}$ be the associated \emph{minor}, i.e.\ the determinant of the $k\times k$ submatrix of~$B$ consisting of entries whose row index is in~$I$ and whose column index is in~$J$. Slightly abusing notation, when $m=n$ we use $B_{i,j}$ as a shorthand for $B_{[n]\setminus\{i\},[n]\setminus\{j\}}$.

As defined in Section~\ref{sec:preliminaries}, for $k\in[\min\{m,n\}]\cup\{0\}$, let $D_k(B)$ be the greatest common divisor of all $k\times k$ minors of~$B$, i.e.
\[
D_k(B)=\gcd\{B_{I,J}:I\subseteq[m],J\subseteq[n],\abs{I}=\abs{J}=k\}.
\]
We define $D_0(B)=1$. For $k\in[\min\{m,n\}]$, let $D_k^*(B)$ be the greatest common divisor of the determinants of all $k\times k$ submatrices of~$B$ that include the last row and the last column, i.e.
\[
D_k^*(B)=\gcd\{B_{I,J}:I\subseteq[m],J\subseteq[n],\abs{I}=\abs{J}=k,m\in I,n\in J\}.
\]
Analogues of all results in this section would hold if we instead defined $D_k^*(B)$ by requiring submatrices to include any specified row and any specified column, but the notion of $D_k^*(B)$ given here is what is of use in Section~\ref{sec:mainresults}.

We first collect some elementary facts about $D_k(B)$ and $D^*_k(B)$.
\begin{proposition}\label{prop:minorfacts}
Let $B$ be an $m\times n$ matrix with integer entries. The following statements hold: 
\begin{enumerate}[label=\textup{(\alph*)},ref=\textup{\alph*}]
    \item $D_k(B)$ divides $D_k^*(B)$ for all $k\in[\min\{m,n\}]$, \label{part:DkdividesDk*}
    \item if $B'$ is an $m'\times n'$ submatrix of~$B$, then $D_k(B)$ divides $D_k(B')$ for all $k\in[\min\{m',n'\}]\cup\{0\}$,\label{part:DkBdividesDkB'}
    \item if $B'$ is an $m'\times n'$ submatrix of~$B$ that contains the last row and the last column, then $D_k^*(B)$ divides $D_k^*(B')$ for all $k\in[\min\{m',n'\}]$,  \label{part:Dk*BdividesDk*B'}
    \item $D_n(B)=D_n^*(B)=\det(B)$ whenever $m=n$, and \label{part:DnB=Dn*B=detB}
    \item $D_k(B)$ divides $D_{k+1}(B)$ for all $k\in[\min\{m,n\}-1]\cup\{0\}$.\label{part:DkdividesDk+1}
\end{enumerate}
\end{proposition}

\begin{proof}
Part~(\ref{part:DkdividesDk*}) follows from the fact that the $k\times k$ submatrices of~$B$ that include the last row and last column form a subset of the full set of $k\times k$ submatrices of~$B$. Part~(\ref{part:DkBdividesDkB'}) follows from the fact that the $k\times k$ submatrices of~$B'$ form a subset of the $k\times k$ submatrices of~$B$. Similarly, for part~(\ref{part:Dk*BdividesDk*B'}), the $k\times k$ submatrices of~$B'$ that include the last row and last column form a subset of the $k\times k$ submatrices of~$B$ that include the last row and last column. For part~(\ref{part:DnB=Dn*B=detB}), when $m=n$ there is only a single $n\times n$ submatrix of~$B$, namely the full matrix~$B$, and this includes the last row and the last column.

For part~(\ref{part:DkdividesDk+1}), let $B'$ be an arbitrary $(k+1)\times(k+1)$ submatrix of~$B$. Expanding along row~$1$, we have
\[
\det(B')=\sum_{j=1}^{k+1}(-1)^{j+1}b'_{1,j}B'_{1,j}.
\]
As $D_k(B)$ divides $B'_{1,j}$ for all $j\in[k+1]$, it also divides $\det(B')$. Since $B'$ is arbitrary, B\'{e}zout's identity then gives that $D_k(B)$ divides $D_{k+1}(B)$.
\end{proof}

\begin{remark}
Proposition~\ref{prop:minorfacts}(\ref{part:DkdividesDk+1}) also follows from the existence of the Smith normal form of~$B$, which guarantees that $\alpha_{k+1}=D_{k+1}(B)/D_k(B)$ is an integer for all $k\in[\min\{m,n\}-1]\cup\{0\}$. We include the elementary proof given above because a similar approach is used in other proofs in this section.
\end{remark}

Before stating the primary results of this section (Proposition~\ref{prop:Dk*dividesDk+1*}, Lemma~\ref{lem:Dk*upperbound}, and Proposition~\ref{prop:D1D2*dividesD1*D2}), we remark that the corresponding results do not hold for determinantal ideals in a polynomial ring where the entries of the matrix are distinct indeterminates, as we have checked using \texttt{SageMath}~\cite{Sage}. Our proofs therefore rely not only on determinantal identities but also on factorization properties of the integers.

We first recall the Desnanot--Jacobi identity \cite{Jacobi}, which says that for an $n\times n$ matrix~$B$ and $i_1,i_2,j_1,j_2\in[n]$ with $i_1<i_2$ and $j_1<j_2$, we have
\[
B_{I\setminus\{i_1,i_2\},J\setminus\{j_1,j_2\}}\det(B)=B_{i_1,j_1}B_{i_2,j_2}-B_{i_1,j_2}B_{i_2,j_1}.
\]
For a proof and historical remarks on the Desnanot--Jacobi identity, we refer the reader to \cite[Theorem~3.12]{Bressoud}. 

\begin{proposition}\label{prop:Dk*dividesDk+1*}
If $B$ is an $m\times n$ matrix with integer entries, then $D_k^*(B)$ divides $D_{k+1}^*(B)$ for all $k\in\{2,3,\dotsc,\min\{m,n\}-1\}$.
\end{proposition}

\begin{proof}
We proceed by induction on~$k$. The base case is $k=2$. Let $B'$ be an arbitrary $3\times3$ submatrix of~$B$ that includes the last row and the last column. Let $p_1^{a_1}p_2^{a_2}\dotsm p_q^{a_q}$ be the prime factorization of $D_2^*(B)$. Fix $\ell\in[q]$. 

If $p_\ell^{a_\ell}$ divides $b'_{3,3}$, we write $\det(B')$ as its expansion along row~$1$ plus its expansion along row~$2$ minus its expansion along row~$3$ to get
\[
\det(B')=b'_{1,1}B'_{1,1}-b'_{1,2}B'_{1,2}-b'_{2,1}B'_{2,1}+b'_{2,2}B'_{2,2}-b'_{3,3}B'_{3,3}.
\]
Since $p_\ell^{a_\ell}$ divides $B'_{1,1}$, $B'_{1,2}$, $B'_{2,1}$, $B'_{2,2}$, and $b'_{3,3}$, it therefore also divides $\det(B')$. If $p_\ell^{a_\ell}$ does not divide $b'_{3,3}$, we use the Desnanot--Jacobi identity to write $b'_{3,3}\det(B')=B'_{1,1}B'_{2,2}-B'_{1,2}B'_{2,1}$. As each of $B'_{1,1}$, $B'_{2,2}$, $B'_{1,2}$, and $B'_{2,1}$ is divisible by $p_\ell^{a_\ell}$, we have that $b'_{3,3}\det(B')$ is divisible by $p_\ell^{2a_\ell}$. Since $b'_{3,3}$ is not divisible by $p_\ell^{a_\ell}$, it follows that $\det(B')$ is divisible by $p_\ell^{a_\ell}$. 

Running this argument for all~$\ell$, we have that $D_2^*(B)$ divides $\det(B')$. Since $B'$ is arbitrary, B\'{e}zout's identity gives that $D_2^*(B)$ divides $D_3^*(B)$. This completes the base case. 

Now fix $k\geq3$. Let $B''$ be an arbitrary $(k+1)\times(k+1)$ submatrix of~$B$ that includes the last row and the last column. Let $p_1^{c_1}p_2^{c_2}\dotsm p_q^{c_q}$ be the prime factorization of $D_k^*(B'')$. Fix $\ell\in[q]$.

By induction, $D_{k-1}^*(B'')$ divides $D_k^*(B'')$, and therefore there exist $i_1,i_2,j_1,j_2\in[k]$ with $i_1<i_2$ and $j_1<j_2$ such that $B''_{[k+1]\setminus\{i_1,i_2\},[k+1]\setminus\{j_1,j_2\}}$ is not divisible by $p_\ell^{c_\ell+1}$.
The Desnanot--Jacobi identity gives 
\[
B''_{[k+1]\setminus\{i_1,i_2\},[k+1]\setminus\{j_1,j_2\}}\det(B'')=B''_{i_1,j_1}B''_{i_2,j_2}-B''_{i_1,j_2}B''_{i_2,j_1}.
\]
As each of $B''_{i_1,j_1}$, $B''_{i_2,j_2}$, $B''_{i_1,j_2}$, and $B''_{i_2,j_1}$ is divisible by $p_\ell^{c_\ell}$, we have that 
\[
B''_{[k+1]\setminus\{i_1,i_2\},[k+1]\setminus\{j_1,j_2\}}\det(B'')
\]
is divisible by $p_\ell^{2c_\ell}$. Since $B''_{[k+1]\setminus\{i_1,i_2\},[k+1]\setminus\{j_1,j_2\}}$ is not divisible by $p_\ell^{c_\ell+1}$, it follows that $\det(B'')$ is divisible by $p_\ell^{c_\ell}$.

Running this argument for all~$\ell$, we have that $D_k^*(B'')$ divides $\det(B'')$. By Proposition~\ref{prop:minorfacts}(\ref{part:Dk*BdividesDk*B'}), $D_k^*(B)$ divides $D_k^*(B'')$, so therefore it divides $\det(B'')$. Since $B''$ is arbitrary, B\'{e}zout's identity gives that $D_k^*(B)$ divides $D_{k+1}^*(B)$.
\end{proof}

The next lemma, which gives a complementary statement to Proposition~\ref{prop:minorfacts}(\ref{part:DkdividesDk*}) about the relationship between $D_k(B)$ and $D_k^*(B)$, is key to our results in Section~\ref{sec:mainresults}. Let $g_i(B)$ denote the greatest common divisor of the entries of row~$i$ of~$B$, i.e.\ $g_i(B)=\gcd_{j\in[n]}(b_{i,j})$, and let $B^T$ denote the transpose of~$B$.

\begin{lemma}\label{lem:Dk*upperbound}
If $B$ is an $m\times n$ matrix with integer entries, then $D_k^*(B)$ is a divisor of $g_n(B^T)g_m(B)D_k(B)$ for all $k\in\{2,3,\dotsc,\min\{m,n\}\}$. 
\end{lemma}

\begin{proof}
Let $I\subseteq[m]$ and $J\subseteq[n]$ with $\abs{I}=\abs{J}=k$. We will first show $D_k^*(B)$ divides $b_{s,n}B_{I,J}$ for all $s\in[m]$ if $m\in I$, then show $D_k^*(B)$ divides $b_{s,n}b_{m,t}B_{I,J}$ for all $s\in[m]$ and $t\in[n]$, and finally show that $D_k^*(B)$ divides $g_n(B^T)g_m(B)D_k(B)$. 

Fix $s\in[m]$ and $t\in[n]$. First suppose $m\in I$. If $n\in J$ then by definition we have that $D_k^*(B)$ divides $B_{I,J}$ and hence $b_{s,n}B_{I,J}$, so suppose $n\notin J$. Let 
\[
I'=I\cup\{s'\} \text{, where } s'=\begin{cases} 
s & \text{ if } s \notin I\\
\text{any element of } [m]\setminus I & \text{ if } s \in I
\end{cases}
\]
and $J'=J\cup\{n\}$. Let $B'$ be the matrix obtained from~$B$ by replacing row~$s'$ by row~$s$. (If $s'=s$ then $B'=B$.) Expanding along row~$s'$, we have
\[
B'_{I',J'}=\sum_{j\in J'}\pm b_{s,j}B_{I,J'\setminus\{j\}}=\pm b_{s,n}B_{I,J}+\sum_{j\in J}\pm b_{s,j}B_{I,J'\setminus\{j\}},
\]
and hence 
\[
b_{s,n}B_{I,J}=\pm B'_{I',J'}+\sum_{j\in J}\pm b_{s,j}B_{I,J'\setminus\{j\}}.
\]
Since $m\in I$ and $n\in(J'\setminus\{j\})$ for all $j\in J$, we know that $D_k^*(B)$ divides $B_{I,J'\setminus\{j\}}$ for all $j\in J$. If $s'\neq s$, then $B'_{I',J'}=0$ and hence is divisible by $D_k^*(B)$. If $s'=s$, then $B'_{I',J'}=B_{I',J'}$ and, since $m\in I'$ and $n\in J'$, it is divisible by $D_{k+1}^*(B)$ and thus, by Proposition~\ref{prop:Dk*dividesDk+1*}, by $D_k^*(B)$. Hence $D_k^*(B)$ divides $b_{s,n}B_{I,J}$.

Now let $I$ and~$J$ be arbitrary. If $m\in I$ then we already have that $D_k^*(B)$ divides $b_{s,n}B_{I,J}$ and hence $b_{s,n}b_{m,t}B_{I,J}$, so suppose $m\notin I$. Let $I''=I\cup\{m\}$ and 
\[
J''=J\cup\{t'\} \text {, where } t'=\begin{cases} 
t & \text{ if } t \notin J\\
\text{any element of } [n]\setminus J & \text{ if } t \in J.
\end{cases}
\]
Let $B''$ be the matrix obtained from~$B$ by replacing column~$t'$ by column~$t$. (If $t'=t$ then $B''=B$.) Expanding along column~$t'$, we have
\[
B''_{I'',J''}=\sum_{i \in I''}\pm b_{i,t} B_{I''\setminus\{i\},J}=\pm b_{m,t}B_{I,J}+\sum_{i\in I}\pm b_{i,t}B_{I''\setminus\{i\},J},
\]
and hence
\[
b_{s,n}b_{m,t}B_{I,J}=\pm b_{s,n}B''_{I'',J''}\pm\sum_{i\in I}\pm b_{s,n}b_{i,t}B_{I''\setminus\{i\},J}.
\]
Let $s''$ be any element of~$I$. Expanding along row~$s''$, we obtain
\[
B''_{I'',J''}=\sum_{j\in J''}\pm b_{s'',j}B_{I''\setminus\{s''\},J''\setminus \{j\}},
\]
and hence
\[
b_{s,n}b_{m,t}B_{I,J}=\sum_{j\in J''}\pm b_{s,n}b_{s'',j}B_{I''\setminus\{s''\},J''\setminus \{j\}}+\sum_{i\in I}\pm b_{s,n}b_{i,t}B_{I''\setminus\{i\},J}.
\]
Since $m\in I''\setminus\{s''\}$ and $m\in I''\setminus\{i\}$ for all $i\in I$, we know from the previous part that $D_k^*(B)$ divides $b_{s,n}B_{I''\setminus\{s''\},J''\setminus \{j\}}$ for all $j\in J$ and $b_{s,n}b_{i,t}B_{I''\setminus\{i\},J}$ for all $i\in I$. Therefore $D_k^*(B)$ divides $b_{s,n}b_{m,t}B_{I,J}$.

Since $D_k^*(B)$ divides $b_{s,n}b_{m,t}B_{I,J}$ for all $s\in[m]$ and $t\in[n]$, using B\'{e}zout's identity gives that $D_k^*(B)$ divides $g_n(B^T)b_{m,t} B_{I,J}$ for all $t\in[n]$ and hence divides $g_n(B^T)g_m(B) B_{I,J}$. Since $I$ and~$J$ are arbitrary, B\'{e}zout's identity gives that $D_k^*(B)$ divides $g_n(B^T)g_m(B)D_k(B)$. 
\end{proof}

The final result of this section is a divisibility property involving both $D_k(B)$ and $D_k^*(B)$ that will be used in Proposition~\ref{prop:alpha'1}(\ref{part:alpha'1upperbound2}) to control the size of~$\alpha_1'$.

\begin{proposition}\label{prop:D1D2*dividesD1*D2}
Fix $m,n\geq2$. If $B$ is an $m\times n$ matrix with integer entries, then $D_1(B)D_{2}^*(B)$ divides $D_1^*(B)D_{2}(B)$.
\end{proposition}

\begin{proof}
Choose arbitrary $I=\{i_1,i_2\}\subseteq[m]$ and $J=\{j_1,j_2\}\subseteq[n]$ with $i_1<i_2$ and $j_1<j_2$. It is enough to show that $D_1(B)D_2^*(B)$ divides $b_{m,n}B_{I,J}$ as B\'{e}zout's identity then implies that it divides $b_{m,n}D_2(B)=D_1^*(B)D_2(B)$. If $m\in I$ and $n\in J$, then by definition $D^*_2(B)$ divides $B_{I,J}$, and so $D_1(B)D^*_2(B)$ divides $b_{m,n}B_{I,J}$. 

For the cases where $m \notin I$ or $n \notin J$, it is useful to construct a $3\times 3$ matrix~$B'$. To do this, first let 
\[
I'=I\cup\{s\} \text {, where } s=\begin{cases} 
m & \text{ if } m \notin I\\
\text{any element of } [m]\setminus I & \text{ if } m \in I,
\end{cases}
\]
and
\[
J'=J\cup\{t\} \text {, where } t=\begin{cases} 
n & \text{ if } n \notin J\\
\text{any element of } [n]\setminus J & \text{ if } n \in J.
\end{cases}
\]
Then let $B'$ be the $3\times 3$ submatrix of~$B$ consisting of entries with row index in~$I'$ and column index in~$J'$. Let $s'\in \{1,2,3\}$ be the index of the row of~$B'$ corresponding to row~$s$ of~$B$, and let $t'\in \{1,2,3\}$ be the index of the column of~$B'$ corresponding to column~$t$ of~$B$. We now proceed with the proof that $D_1(B)D_2^*(B)$ divides $b_{m,n}B_{I,J}$ in the remaining cases.

If $m\in I$ but $n\notin J$, let $B''$ be the matrix obtained from~$B'$ by replacing row~$s'$ by row~$3$. Expanding along row~$s'$ gives 
\[\begin{split}
0=\det(B'')&=\pm (b'_{3,1}B'_{s',1}-b'_{3,2}B'_{s',2}+b'_{3,3}B'_{s',3})\\
&=\pm (b_{m,j_1}B_{I,\{j_2,n\}}- b_{m,j_2}B_{I,\{j_1,n\}}+b_{m,n}B_{I,J}),
\end{split}\]
and hence
\[
b_{m,n}B_{I,J}=-b_{m,j_1}B_{I,\{j_2,n\}}+b_{m,j_2}B_{I,\{j_1,n\}}.
\]
Since $m\in I$, this shows that $b_{m,n}B_{I,J}$ is divisible by $D_1(B)D_2^*(B)$. If $m\notin I$ but $n\in J$, we apply the same argument to the transpose of~$B'$ to get that $D_1(B)D_2^*(B)$ divides $b_{m,n}B_{I,J}$.

The remaining case is that $m\notin I$ and $n\notin J$. Let $p_1^{a_1}p_2^{a_2}\dotsm p_q^{a_q}$ be the prime factorization of $D_1(B)D_2^*(B)$. Fix $\ell\in [q]$. Let $c_\ell$ be the largest integer for which $p_\ell^{c_\ell}$ divides $D_1(B)$ and let $e_\ell$ be the largest integer for which $p_\ell^{e_\ell}$ divides $D_2^*(B)$. Note that $a_\ell=c_\ell+e_\ell$ for all~$\ell$. If $p_\ell^{e_\ell-c_\ell}$ divides $b_{m,n}$, then because $p_\ell^{c_\ell}$ divides every entry of~$B'$ we have that $p_\ell^{2c_\ell}$ divides $B'_{3,3}$ and hence that $p_\ell^{e_\ell-c_\ell+2c_\ell}=p_\ell^{a_\ell}$ divides $b_{m,n}B'_{3,3}=b_{m,n}B_{I,J}$. If $p_\ell^{e_\ell-c_\ell}$ does not divide $b_{m,n}$, we use the Desnanot--Jacobi identity to write $b_{m,n}\det(B')=b'_{3,3}\det(B')=B'_{1,1}B'_{2,2}-B'_{1,2}B'_{2,1}$. Since each of $B'_{1,1}$, $B'_{2,2}$, $B'_{1,2}$, and $B'_{2,1}$ is divisible by $p_\ell^{e_\ell}$, we have that $b_{m,n}\det(B')$ is divisible by $p_\ell^{2e_\ell}$. As $b_{m,n}$ is not divisible by $p_\ell^{e_\ell-c_\ell}$, it follows that $\det(B')$ is divisible by $p_\ell^{2e_\ell-(e_\ell-c_\ell)}=p_\ell^{a_\ell}$. Now we write $\det(B')$ as its expansion along the row~$1$ plus its expansion along the row~$2$ minus its expansion along row~$3$ to get
\[\begin{split}
\det(B')&=b'_{1,1}B'_{1,1}-b'_{1,2}B'_{1,2}-b'_{2,1}B'_{2,1}+b'_{2,2}B'_{2,2}-b'_{3,3}B'_{3,3}\\
&=b_{i_1,j_1}B'_{1,1}-b_{i_1,j_2}B'_{1,2}-b_{i_2,j_1}B'_{2,1}+b_{i_2,j_2}B'_{2,2}-b_{m,n}B_{I,J}.
\end{split}\]
It follows that
\[
b_{m,n}B_{I,J}=b_{i_1,j_1}B'_{1,1}-b_{i_1,j_2}B'_{1,2}-b_{i_2,j_1}B'_{2,1}+b_{i_2,j_2}B'_{2,2}-\det(B').
\]
Since each term of the right side is divisible by $p_\ell^{a_\ell}$, so is $b_{m,n}B_{I,J}$. Running this argument for all~$\ell$ gives that $D_1(B)D_2^*(B)$ divides $b_{m,n}B_{I,J}$. 
\end{proof}

We believe an analogue of Proposition~\ref{prop:D1D2*dividesD1*D2} for larger minors holds and state this as Conjecture~\ref{conj:minors}.

\section{Critical groups under the Keyes--Reiter operation}\label{sec:mainresults}

We now determine how critical groups are transformed under the Keyes--Reiter generalized star-clique operation described in Subsection~\ref{subsec:operation}. Recall that $(G',\vd',\vr')$ denotes the result of performing the operation on $(G,\vd,\vr)$ at the vertex~$v_n$ and that we let $L=\diag(\vd)-A$ and $L'=\diag(\vd')-A'$. The approach of this section is to establish a relationship between minors of~$L'$ and certain minors of~$L$ and use this to deduce information about the invariant factors of~$L'$. This then gives information about the relationship between the critical groups $\mathcal{K}(G,\vd,\vr)$ and $\mathcal{K}(G',\vd',\vr')$.

We first record a result of Chio~\cite{Chio}\footnote{We were unable to obtain Chio's manuscript~\cite{Chio} but read the summary of it in~\cite{Muir}. For comments about the rendering of Chio's name, see \cite[footnote~2]{Heinig}. For further historical background, see \cite[Section~2]{Abeles}.} that underlies the relationship between minors of~$L$ and minors of~$L'$. For the convenience of the reader and to keep this article self contained, we include a short proof, following \cite[Theorem~3.6.1]{Eves}. 

\begin{lemma}[Chio~\cite{Chio}]\label{lem:Chio}
Fix $n\geq2$, and let $B$ be an $n\times n$ matrix with entries $b_{i,j}$. If $B'$ is the $(n-1)\times(n-1)$ matrix with $(i,j)$ entry $b_{i,j}b_{n,n}-b_{i,n}b_{n,j}$, then $\det(B')=b_{n,n}^{n-2}\det(B)$.
\end{lemma}

\begin{proof}
When $n=2$, the result is immediately true. When $n\geq3$ and $b_{n,n}=0$, all columns of~$B'$ are scalar multiples of the same vector, and hence we have that $\det(B')=0=b_{n,n}^{n-2}\det(B)$. When $n\geq3$ and $b_{n,n}\neq0$, multiply each of row~$1$ through row $n-1$ of~$B$ by $b_{n,n}$ and then subtract $b_{i,n}$ times row~$n$ from row~$i$ for all $i\in[n-1]$. This gives 
\[\begin{split}
b_{n,n}^{n-1}\det(B)&=\scriptsize{\begin{vmatrix}
b_{1,1}b_{n,n}-b_{1,n}b_{n,1} & \cdots & b_{1,n-1}b_{n,n}-b_{1,n}b_{n,n-1} & 0\\
\vdots & \ddots & \vdots & \vdots\\
b_{n-1,1}b_{n,n}-b_{n-1,n}b_{n,1} & \cdots & b_{n-1,n-1}b_{n,n}-b_{n-1,n}b_{n,n-1} & 0\\
b_{n,1} & \cdots & b_{n,n-1} & b_{n,n}
\end{vmatrix}}\\
&=b_{n,n}\det(B').
\end{split}\]
Canceling $b_{n,n}$ gives the desired result. 
\end{proof}

\begin{remark}
The Desnanot--Jacobi identity (stated before Proposition~\ref{prop:Dk*dividesDk+1*}) and Chio pivotal condensation (Lemma~\ref{lem:Chio}) are special cases of Sylvester's determinant identity.\footnote{This identity was first stated, without proof, by Sylvester in~\cite{Sylvester}. For the statement and several proofs, see~\cite{AAM} and the references therein.}
\end{remark}

As in the previous section, let $D_k(L)$ denote the greatest common divisor of all $k\times k$ minors of~$L$, let $D_k^*(L)$ denote the greatest common divisor of all $k\times k$ minors of~$L$ that include the last row and the last column, and let $g_n(L)$ denote the greatest common divisor of the entries of the last row of~$L$. We now establish relationships between $D_k(L')$ and $D_{k+1}(L)$ that will lead to connections between the critical groups of $(G,\vd,\vr)$ and $(G',\vd',\vr')$ in Corollaries \ref{cor:orderbound} and \ref{cor:gcd=1}.

\begin{theorem}\label{thm:DkL'}
For all $k\in[n-2]$, 
\begin{enumerate}[label=\textup{(\alph*)},ref=\textup{\alph*}]
    \item $D_{k}(L')=d_n^{k-1}D_{k+1}^*(L)$, \label{part:DkL'equality}
    \item $d_n^{k-1}D_{k+1}(L)$ divides $D_{k}(L')$, \label{part:DkL'lowerbound}
    \item $D_{k}(L')$ divides $(g_n(L))^2d_n^{k-1}D_{k+1}(L)$, and \label{part:DkL'upperbound}
    \item $d_n^{k-1}D_{k+1}(L)\leq D_{k}(L')\leq (g_n(L))^2d_n^{k-1}D_{k+1}(L)$.\label{part:DkL'bounds}
\end{enumerate}
\end{theorem}

\begin{proof}
There is a one-to-one correspondence between $k\times k$ submatrices of the matrix~$L'$ and $(k+1)\times(k+1)$ submatrices of~$L$ that include the last row and the last column, where the submatrix of~$L'$ with row indices in~$I$ and column indices in~$J$ is associated with the submatrix of~$L$ with row indices in $I\cup\{n\}$ and column indices in $J\cup\{n\}$. Recall from Subsection~\ref{subsec:operation} that $\ell'_{i,j}=\ell_{i,j}\ell_{n,n}-\ell_{i,n}\ell_{n,j}$ for all $i,j\in[n-1]$. If $B'$ is a $k\times k$ submatrix of~$L'$ and $B$ is the associated $(k+1)\times(k+1)$ submatrix of~$L$, this expression for the entries of~$L'$ implies that $b'_{i,j}=b_{i,j}b_{k+1,k+1}-b_{i,k+1}b_{k+1,j}$ for all $i,j\in[k]$. Applying Lemma~\ref{lem:Chio} then gives 
\[
\det(B')=b_{k+1,k+1}^{k-1}\det(B)=\ell_{n,n}^{k-1}\det(B)=d_n^{k-1}\det(B).
\]
Taking the greatest common denominator over all pairs $B'$ and~$B$, we get $D_{k}(L')=d_n^{k-1}D_{k+1}^*(L)$, proving~(\ref{part:DkL'equality}). 

To prove (\ref{part:DkL'lowerbound}), Proposition~\ref{prop:minorfacts}(\ref{part:DkdividesDk*}) implies that $d_n^{k-1}D_{k+1}(L)$ divides $d_n^{k-1}D^*_{k+1}(L)$, which, by part~(\ref{part:DkL'equality}), equals $D_{k}(L')$. To prove (\ref{part:DkL'upperbound}), part~(\ref{part:DkL'equality}) gives that $D_{k}(L')=d_n^{k-1}D_{k+1}^*(L)$, which, by Lemma~\ref{lem:Dk*upperbound}, divides $g_n(L^T)g_n(L)d_n^{k-1}D_{k+1}(L)$. Since $L$ is symmetric, we have $g_n(L^T)g_n(L)=(g_n(L))^2$, completing the proof of~(\ref{part:DkL'upperbound}). Part~(\ref{part:DkL'bounds}) immediately follows from (\ref{part:DkL'lowerbound}) and~(\ref{part:DkL'upperbound}).
\end{proof}

Focusing on the special case $k=n-2$, Theorem~\ref{thm:DkL'} gives information about how the order of the critical group is transformed by the operation.

\begin{corollary}\label{cor:orderbound}
The following statements hold:
\begin{enumerate}[label=\textup{(\alph*)},ref=\textup{\alph*}]
\item $d_n^{n-3}\abs{\mathcal{K}(G;\vd,\vr)}$ divides $\abs{\mathcal{K}(G';\vd',\vr')}$,
\item $\abs{\mathcal{K}(G';\vd',\vr')}$ divides $(g_n(L))^2d_n^{n-3}\abs{\mathcal{K}(G;\vd,\vr)}$, and\label{part:orderupperbound}
\item $d_n^{n-3}\abs{\mathcal{K}(G;\vd,\vr)}\leq \abs{\mathcal{K}(G';\vd',\vr')}\leq (g_n(L))^2d_n^{n-3}\abs{\mathcal{K}(G;\vd,\vr)}$.\label{part:orderbound}
\end{enumerate}
\end{corollary}

\begin{proof}
As discussed in Subsection~\ref{subsec:criticalgroups}, $\abs{\mathcal{K}(G;\vd,\vr)}=D_{n-1}(L)$ and $\abs{\mathcal{K}(G';\vd',\vr')}=D_{n-2}(L')$. The result then immediately follows by applying Theorem~\ref{thm:DkL'} with $k=n-2$.
\end{proof}

Whenever $g_n(L)=1$, Corollary~\ref{cor:orderbound}(\ref{part:orderbound}) says that $\abs{\mathcal{K}(G';\vd',\vr')}=d_n^{n-3}\abs{\mathcal{K}(G;\vd,\vr)}$, i.e.\ that the order of $\mathcal{K}(G';\vd',\vr')$ is determined exactly. This applies, in particular, whenever $G$ is a connected, simple graph, as then $\ell_{n,j}=1$ for some $j\in [n-1]$, implying that $g_n(L)=1$. When $g_n(L)=1$, we moreover can precisely determine the group $\mathcal{K}(G';\vd',\vr')$ in terms of $\mathcal{K}(G;\vd,\vr)$, which we do in Corollary~\ref{cor:gcd=1}.
 
When $g_n(L)\neq 1$, Example~\ref{ex:nonsimple} provides an example of a non-simple graph with two arithmetical structures that share the same critical group and the same value of~$d_n$, but after performing the operation they achieve the two distinct bounds in Corollary~\ref{cor:orderbound}(\ref{part:orderbound}).

We now turn our attention to understanding the relationship between $\mathcal{K}(G;\vd,\vr)$ and $\mathcal{K}(G';\vd',\vr')$ more precisely. Recall from Subsection~\ref{subsec:criticalgroups} that $\mathcal{K}(G;\vd,\vr)\cong\bigoplus_{k=1}^{n-1}\mathbb{Z}/\alpha_k\mathbb{Z}$, where the invariant factors $\alpha_k=D_k(L)/D_{k-1}(L)$ are the diagonal entries of the Smith normal form of~$L$. Similarly, $\mathcal{K}(G';\vd',\vr')\cong\bigoplus_{k=1}^{n-2}\mathbb{Z}/\alpha'_k\mathbb{Z}$, where $\alpha'_k=D_k(L')/D_{k-1}(L')$. We begin by establishing some relationships between $\alpha_1'$ and~$\alpha_2$.

\begin{proposition}\label{prop:alpha'1}
The following statements hold:
\begin{enumerate}[label=\textup{(\alph*)},ref=\textup{\alph*}]
    \item $\alpha'_1$ divides $(g_n(L))^2\alpha_1\alpha_2$,\label{part:alpha'1upperbound1}
    \item $\alpha'_1$ divides $d_n\alpha_2$,\label{part:alpha'1upperbound2}
    \item $\alpha'_1$ divides $\gcd((g_n(L))^2\alpha_1,d_n)\alpha_2$,\label{part:alpha'1upperbound3}
    \item $\alpha_1\alpha_2$ divides $\alpha'_1$, and\label{part:alpha'1lowerbound}
    \item $\alpha_1\alpha_2\leq\alpha'_1\leq\gcd((g_n(L))^2\alpha_1,d_n)\alpha_2$.\label{part:alpha'1bounds}
\end{enumerate}
\end{proposition}

\begin{proof}
For~(\ref{part:alpha'1upperbound1}), we have 
\[
\alpha'_1=\frac{D_1(L')}{D_0(L')}=D_1(L').
\]
By Theorem~\ref{thm:DkL'}(\ref{part:DkL'upperbound}), this divides $(g_n(L))^2D_2(L)=(g_n(L))^2\alpha_1\alpha_2$.

For~(\ref{part:alpha'1upperbound2}), we have 
\[
\alpha'_1=\frac{D_1(L')}{D_0(L')}=D_1(L').
\]
By Theorem~\ref{thm:DkL'}(\ref{part:DkL'equality}), this equals $D_2^*(L)$, which, by Proposition~\ref{prop:D1D2*dividesD1*D2}, divides
\[
\frac{D_1^*(L)D_2(L)}{D_1(L)}=d_n\alpha_2.
\]

Part~(\ref{part:alpha'1upperbound3}) follows from (\ref{part:alpha'1upperbound1}) and~(\ref{part:alpha'1upperbound2}) using B\'{e}zout's identity. 

For~(\ref{part:alpha'1lowerbound}), we have $\alpha_1\alpha_2=D_2(L)$. By Proposition~\ref{prop:minorfacts}(\ref{part:DkdividesDk*}), this divides $D_2^*(L)$, which, by Theorem~\ref{thm:DkL'}(\ref{part:DkL'equality}), equals $D_1(L')=\alpha_1'$. 

Part~(\ref{part:alpha'1bounds}) immediately follows from (\ref{part:alpha'1upperbound3}) and~(\ref{part:alpha'1lowerbound}).
\end{proof}

We next establish some relationships between $\alpha'_k$ and $\alpha_{k+1}$ for $k\geq2$.

\begin{theorem}\label{thm:alpha'k}
For all $k\in\{2,3,\dotsc,n-2\}$,
\begin{enumerate}[label=\textup{(\alph*)},ref=\textup{\alph*}]
\item $\alpha'_k$ divides $(g_n(L))^2d_n\alpha_{k+1}$, \label{part:alpha'kupperbound}
\item $d_n\alpha_{k+1}$ divides $(g_n(L))^2\alpha'_k$, and \label{part:alpha'klowerbound}
\item $\dfrac{d_n\alpha_{k+1}}{(g_n(L))^2}\leq\alpha'_k\leq(g_n(L))^2d_n\alpha_{k+1}$.\label{part:alpha'kbounds}
\end{enumerate}
\end{theorem}

\begin{proof}
We use Theorem~\ref{thm:DkL'}(\ref{part:DkL'equality}) to get
\[
\alpha_k'=\frac{D_k(L')}{D_{k-1}(L')}=\frac{D_k(L')}{d_n^{k-2} D_{k}^*(L)},
\]
which, by Theorem~\ref{thm:DkL'}(\ref{part:DkL'upperbound}), divides 
\[
\frac{(g_n(L))^2d_n^{k-1} D_{k+1}(L)}{d_n^{k-2} D^*_{k}(L)}.
\]
By Proposition~\ref{prop:minorfacts}(\ref{part:DkdividesDk*}), this divides 
\[
\frac{(g_n(L))^2d_n^{k-1} D_{k+1}(L)}{d_n^{k-2} D_{k}(L)}=(g_n(L))^2d_n\alpha_{k+1},
\]
completing the proof of~(\ref{part:alpha'kupperbound}).

For~(\ref{part:alpha'klowerbound}), we have
\[
d_na_{k+1}=d_n\frac{D_{k+1}(L)}{D_k(L)}=\frac{(g_n(L))^2d_n^{k-1}D_{k+1}(L)}{(g_n(L))^2d_n^{k-2}D_k(L)},
\]
which, by Lemma~\ref{lem:Dk*upperbound}, divides 
\[
\frac{(g_n(L))^2d_n^{k-1}D_{k+1}(L)}{d_n^{k-2}D_k^*(L)}.
\]
By Proposition~\ref{prop:minorfacts}(\ref{part:DkdividesDk*}), this divides 
\[
\frac{(g_n(L))^2d_n^{k-1}D_{k+1}^*(L)}{d_n^{k-2}D_k^*(L)}=(g_n(L))^2\frac{D_k(L')}{D_{k-1}(L')}=(g_n(L))^2\alpha'_k,
\]
where the first equality follows by Theorem~\ref{thm:DkL'}(\ref{part:DkL'equality}).

Part~(\ref{part:alpha'kbounds}) immediately follows from (\ref{part:alpha'kupperbound}) and~(\ref{part:alpha'klowerbound}).
\end{proof}

Whenever $g_n(L)=1$, and in particular when $G$ is a simple graph, we can determine $\alpha'_k$ exactly. 

\begin{corollary}\label{cor:gcd=1}
If $g_n(L)=1$ \textup{(}for example, if $G$ is a simple graph\textup{)}, then $\alpha'_1=\alpha_2$ and $\alpha'_k=d_n\alpha_{k+1}$ for all $k\in \{2,3,\dotsc, n-2\}$. In particular,
\[
\mathcal{K}(G';\vd',\vr')\cong\mathbb{Z}/\alpha_{2}\mathbb{Z}\oplus \Biggl(\bigoplus_{k=2}^{n-2}\mathbb{Z}/d_n\alpha_{k+1}\mathbb{Z}\Biggr).
\]
\end{corollary}

\begin{proof}
When $g_n(L)=1$, Proposition~\ref{prop:alpha'1}(\ref{part:alpha'1bounds}) immediately gives that $\alpha'_1=\alpha_1\alpha_2$. In this case we also have that $\alpha_1=D_1(L)=1$, so therefore $\alpha'_1=\alpha_2$. For all $k\in \{2,3,\dotsc, n-2\}$, Theorem~\ref{thm:alpha'k}(\ref{part:alpha'kbounds}) gives that $\alpha'_k=d_n\alpha_{k+1}$. 

When $G$ is a connected, simple graph, $\ell _{n,j}=1$ for some $j\in [n-1]$, and hence $g_{n}(L)=1$. Therefore this result applies to all arithmetical structures on connected, simple graphs.
\end{proof}

\begin{remark}
In the special case when $G$ is a simple graph and $d_n=1$, we have that $\alpha_1=D_1(L)=1$, and therefore Corollary~\ref{cor:gcd=1} indicates that the star-clique operation preserves the critical group, consistent with the observation in \cite[Section~1.8]{L89}. 
\end{remark}

Since $\abs{\mathcal{K}(G';\vd',\vr')}=\alpha'_1\alpha'_2\dotsm\alpha'_{n-2}$ and $\abs{\mathcal{K}(G;\vd,\vr)}=\alpha_1\alpha_2\dotsm\alpha_{n-1}$, Corollary~\ref{cor:orderbound}(\ref{part:orderbound}) implies that
\[
\alpha'_1\alpha'_2\dotsm\alpha'_{n-2}\leq (g_n(L))^2d_n^{n-3}\alpha_1\alpha_2\dotsm\alpha_{n-1}.
\]
This suggests it may be possible to improve the bound $\alpha'_k\leq(g_n(L))^2d_n\alpha_{k+1}$ in Theorem~\ref{thm:alpha'k}(\ref{part:alpha'kbounds}) when $g_n(L)\neq1$. Indeed, we conjecture that $\alpha'_k$ divides $d_n\alpha_{k+1}$ for all $k\in[n-2]$.

\begin{conjecture}\label{conj:betterbound}
For all $k\in[n-2]$, $\alpha'_k$ divides $d_n\alpha_{k+1}$.
\end{conjecture}

Conjecture~\ref{conj:betterbound} would imply that $\abs{\mathcal{K}(G';\vd',\vr')}=\alpha'_1\alpha'_2\dotsm\alpha'_{n-2}$ divides
\[
d_n^{n-2}\alpha_2\dotsm\alpha_{n-1}=\frac{{d_n^{n-2}\abs{\mathcal{K}(G;\vd,\vr)}}}{{\alpha_1}}.
\]
With Corollary~\ref{cor:orderbound}(\ref{part:orderupperbound}) and B\'{e}zout's identity, this would show that $\abs{\mathcal{K}(G';\vd',\vr')}$ divides 
\[
\gcd\biggl((g_n(L))^2,\frac{d_n}{\alpha_1}\biggr)d_n^{n-3}\abs{\mathcal{K}(G;\vd,\vr)}.
\]
This would be an improvement of Corollary~\ref{cor:orderbound}(\ref{part:orderupperbound}) whenever $(g_n(L))^2$ does not divide $d_n/\alpha_1$, or equivalently whenever $(g_n(L))^2\alpha_1$ does not divide~$d_n$.

Since
\[
\alpha_k'=\frac{D_k(L')}{D_{k-1}(L')}=\frac{d_n^{k-1}D_{k+1}^*(L)}{d_n^{k-2} D_{k}^*(L)}=d_n\frac{D_{k+1}^*(L)}{ D_{k}^*(L)}
\]
and
\[
d_n\alpha_{k+1}=d_n\frac{D_{k+1}(L)}{D_{k}(L)},
\]
we have that Conjecture~\ref{conj:betterbound} would follow from proving that ${D_{k+1}^*(L)}/{ D_{k}^*(L)}$ divides ${D_{k+1}(L)}/{D_{k}(L)}$, or equivalently that $D_k(L)D_{k+1}^*(L)$ divides $D_{k+1}(L)D_{k}^*(L)$. We believe this statement is true for all matrices with integer entries and record this as another conjecture, as it may be of independent interest to some readers. 

\begin{conjecture}\label{conj:minors}
If $B$ is an $m\times n$ matrix with integer entries, then $D_k(B)D_{k+1}^*(B)$ divides $D_k^*(B)D_{k+1}(B)$ for all $k\in[\min\{m,n\}-1]$. 
\end{conjecture}

Note that Proposition~\ref{prop:D1D2*dividesD1*D2} proves Conjecture~\ref{conj:minors} in the special case $k=1$. Note also that, in the special case where $m=n$ and $k=n-1$, we have by Proposition~\ref{prop:minorfacts}(\ref{part:DnB=Dn*B=detB}) that $D_{n}^*(B)=D_{n}(B)=\det(B)$, and therefore the conjecture reduces to saying that $D_{n-1}(B)$ divides $D^*_{n-1}(B)$, which is true by Proposition~\ref{prop:minorfacts}(\ref{part:DkdividesDk*}).

\section*{Acknowledgments}
We would like to thank Darij Grinberg and Irena Swanson for helpful conversations about the results of Section~\ref{sec:minors}. We also thank Carlo Beenakker, Darij Grinberg, and the MathOverflow user Quizzical for suggesting proofs of Lemma~\ref{lem:Chio} that allowed us to simplify our original argument. Joel Louwsma would like to thank Villanova University and Alexander Diaz-Lopez for their hospitality during visits in which much of the work described in this paper was completed. Alexander Diaz-Lopez was partially supported by Villanova's University Summer Grant (USG). Joel Louwsma was partially supported by two Niagara University Summer Research Awards.

\bibliographystyle{amsplain}
\bibliography{CriticalGroupOperation}

\providecommand{\bysame}{\leavevmode\hbox to3em{\hrulefill}\thinspace}
\providecommand{\MR}{\relax\ifhmode\unskip\space\fi MR }
\providecommand{\MRhref}[2]{%
  \href{http://www.ams.org/mathscinet-getitem?mr=#1}{#2}
}
\providecommand{\href}[2]{#2}
\begin{thebibliography}{10}

\bibitem{Abeles}
Francine~F. Abeles, \emph{Chi\`o's and {D}odgson's determinantal identities},
  Linear Algebra Appl. \textbf{454} (2014), 130--137. \MR{3208413}

\bibitem{AAM}
Alkiviadis~G. Akritas, Evgenia~K. Akritas, and Genadii~I. Malaschonok,
  \emph{Various proofs of {S}ylvester's (determinant) identity}, Math. Comput.
  Simulation \textbf{42} (1996), no.~4-6, 585--593. \MR{1430843}

\bibitem{A20}
Kassie Archer, Abigail~C. Bishop, Alexander Diaz-Lopez, Luis~D.
  Garc\'{\i}a~Puente, Darren Glass, and Joel Louwsma, \emph{Arithmetical
  structures on bidents}, Discrete Math. \textbf{343} (2020), no.~7, Paper No.
  111850, 23 pp. \MR{4072950}

\bibitem{B18}
Benjamin Braun, Hugo Corrales, Scott Corry, Luis~David Garc\'{i}a~Puente,
  Darren Glass, Nathan Kaplan, Jeremy~L. Martin, Gregg Musiker, and Carlos~E.
  Valencia, \emph{Counting arithmetical structures on paths and cycles},
  Discrete Math. \textbf{341} (2018), no.~10, 2949--2963. \MR{3843283}

\bibitem{Bressoud}
David~M. Bressoud, \emph{Proofs and confirmations: The story of the alternating
  sign matrix conjecture}, MAA Spectrum, Mathematical Association of America,
  Washington, DC; Cambridge University Press, Cambridge, 1999. \MR{1718370}

\bibitem{Chio}
M.~F\'{e}lix Chio, \emph{M\'{e}moire sur les fonctions connues sous le nom de
  r\'{e}sultantes ou de d\'{e}terminans [sic]}, A.\ Pons et C, Turin, 1853.

\bibitem{C18}
Hugo Corrales and Carlos~E. Valencia, \emph{Arithmetical structures on graphs},
  Linear Algebra Appl. \textbf{536} (2018), 120--151. \MR{3713448}

\bibitem{CP}
Scott Corry and David Perkinson, \emph{Divisors and sandpiles: {A}n
  introduction to chip-firing}, American Mathematical Society, Providence, RI,
  2018. \MR{3793659}

\bibitem{Eves}
Howard Eves, \emph{Elementary matrix theory}, Allyn and Bacon, Inc., Boston,
  MA, 1966. \MR{0199198}

\bibitem{GK}
Darren Glass and Nathan Kaplan, \emph{Chip-firing games and critical groups}, A
  project-based guide to undergraduate research in mathematics---starting and
  sustaining accessible undergraduate research, Found. Undergrad. Res. Math.,
  Birkh\"{a}user/Springer, Cham, 2020, pp.~107--152. \MR{4291925}

\bibitem{GW}
Darren Glass and Joshua Wagner, \emph{Arithmetical structures on paths with a
  doubled edge}, Integers \textbf{20} (2020), Paper No. A68, 18 pp.
  \MR{4142505}

\bibitem{Heinig}
Peter~Christian Heinig, \emph{Chio condensation and random sign matrices},
  \href{https://arxiv.org/abs/1103.2717v3}{arXiv:1103.2717v3} [math.CO].

\bibitem{Jacobi}
C.~G.~J. Jacobi, \emph{De binis quibuslibet functionibus homogeneis secundi
  ordinis per substitutiones lineares in alias binas tranformandis, quae solis
  quadratis variabilium constant; una cum variis theorematis de tranformatione
  et determinatione integralium multiplicium}, J. Reine Angew. Math.
  \textbf{12} (1834), 1--69. \MR{1577999}

\bibitem{KR}
Christopher Keyes and Tomer Reiter, \emph{Bounding the number of arithmetical
  structures on graphs}, Discrete Math. \textbf{344} (2021), no.~9, Paper No.
  112494, 11 pp. \MR{4271619}

\bibitem{K19}
Caroline~J. Klivans, \emph{The mathematics of chip-firing}, Discrete
  Mathematics and its Applications, CRC Press, Boca Raton, FL, 2019.
  \MR{3889995}
  
\bibitem{L24}
Dino Lorenzini, \emph{The critical polynomial of a graph}, J. Number Theory
  \textbf{257} (2024), 215--248. \MR{4673627}

\bibitem{L08}
\bysame, \emph{Smith normal form and {L}aplacians}, J. Combin. Theory
  Ser. B \textbf{98} (2008), no.~6, 1271--1300. \MR{2462319}

\bibitem{L89}
Dino~J. Lorenzini, \emph{Arithmetical graphs}, Math. Ann. \textbf{285} (1989),
  no.~3, 481--501. \MR{1019714}

\bibitem{Muir}
Thomas Muir, \emph{The theory of determinants in the historical order of
  development, {V}olume {II}}, Macmillan and Company, London, 1911.

\bibitem{Sage}
The Sage~Developers, \emph{{S}age{M}ath, the {S}age {M}athematics {S}oftware
  {S}ystem ({V}ersion 9.6)}, 2022, \url{https://www.sagemath.org}.

\bibitem{SNF}
Richard~P. Stanley, \emph{Smith normal form in combinatorics}, J. Combin.
  Theory Ser. A \textbf{144} (2016), 476--495. \MR{3534076}

\bibitem{Sylvester}
J.~J. Sylvester, \emph{{O}n the relation between the minor determinants of
  linearly equivalent quadratic functions}, Philos. Mag. (4) \textbf{1} (1851),
  no.~4, 295--305.

\end{thebibliography}

\end{document}